\documentclass[pdf]{article}
\usepackage{pst-node}
\usepackage{amssymb}
\usepackage{tikz-cd} 
\usepackage{amsmath}
\usepackage{amsthm}
\usepackage{pgfplots}
\pgfplotsset{width=10cm,compat=1.9}
\usepackage{graphicx} 
\usepackage{listings}
\usepackage{hyperref}
\usepackage{xcolor}
\usepackage[backend=biber,style=alphabetic,sorting=ynt]{biblatex}

\newtheorem{theorem}{Theorem}[section]
\newtheorem{definition}{Definition}[section]

\newtheorem{example}{Example}[theorem]
\newtheorem{conjecture}{Conjecture}[theorem]
\newtheorem{proposition}{Proposition}[theorem]
\newtheorem{lemma}[theorem]{Lemma}
\newtheorem{note}[theorem]{Note}

\hypersetup{ linkcolor=blue }

\makeatletter
\newcommand{\email}[1]{\ttfamily#1}
\makeatother

\begin{document}
\title{A Variational Approach to the Yamabe Problem: Conformal Transformations and Scalar Curvature on Compact Riemannian Manifolds}
\author{Aoran Chen\thanks{\email{chen5531@umn.edu}}}
\date{\vspace{-1em}\normalsize{University of Minnesota--Twin Cities \\ Department of Mathematics \\ Minneapolis, MN}}
\maketitle

\section{ Introduction and preliminaries}
The Yamabe problem asks if any Riemannian metric $g$ on a compact smooth manifold $M$ of dimension $n \geq 3$ is conformal to a metric with constant scalar curvature. 
This problem was born in 1960, from Hidehiko Yamabe's attempt to solve the then not yet proved Poincaré conjecture in his paper \cite{yamabe}.

\begin{conjecture}[Poincaré]
Every simply connected, closed 3-manifold is homeomorphic to the 3-sphere. Where closed means compact without boundary.
\end{conjecture}

In 2003 John Milnor wrote a nice survey \cite{milnor} right after Perelman proved the Poincaré conjecture. The Poincaré conjecture is a purely topological statement, but turned out that it was proved by using the Riemannian metric structure on smooth manifolds. \\

The classical result by Killing and Hopf shows how the metric structure can give topological information, and it provides a road to the Poincaré conjecture.

\begin{theorem}[Killing-Hopf]\label{Killing-Hopf}
The universal cover of a manifold of constant sectional curvature is one of the model spaces: 
\begin{enumerate}
\item sphere (positive sectional curvature)
\item plane (zero sectional curvature)
\item hyperbolic manifold (negative sectional curvature)
\end{enumerate}
\end{theorem}
(See \ref{sectional curvature} for the definition of sectional curvature.)\\

From the above theorem we understand that it suffices to prove the Poincaré conjecture if, over on any simply connected closed 3-manifold, we can build a Riemannian metric with constant sectional curvature. It is a fundamental result that any smooth manifold has a Riemannian metric (see Definition 2.1 in the book \cite{CarmoR}), but the question is whether such a Riemannian metric can have constant sectional curvature? Or, for an arbitrary Riemannian metric, are we able to deform this metric such that it has constant sectional curvature? \\

An approach to finding such a constant curvature metric is to identify the critical point of the Hilbert-Einstein action.

\begin{definition}[Hilbert-Einstein action]\label{Hilbert-Einstein action 1} For a closed Riemannian manifold, the action is given as
\begin{align*}
    \int_M RdV_g,  
\end{align*}

where $R$ and $dV_g$ stand, respectively, for the scalar curvature (see \ref{ricci and scalar} for the definition of scalar curvature) of $M$ and the volume form determined by the metric and orientation. 
\end{definition}
The survey \cite{logunov2004were} provided a historical review on the finding of the Hilbert-Einstein equations. Following the idea of the least action principle (Hilbert's Axiom I), Hilbert found the critical point of the action, which gives us the well-known gravitational field equation.\\


If $M,g$ is the manifold constrained to metrics of volume one, the critical points of Hilbert-Einstein action must satisfy:
\begin{align}
    R_{\mu \nu} = kg_{\mu \nu}, \label{Einstein condition}
\end{align}
where $g_{\mu \nu}$ is the metric tensor, $R_{\mu \nu}$ is the Ricci curvature (see \ref{ricci and scalar} for the definition of the Ricci curvature) of $g_{\mu \nu},$ and $k$ is the constant of proportionality. 

\begin{definition}[Einstein Manifold and Einstein metric]
For any Riemannian manifold whose Ricci tensor is proportional to the metric, it is an Einstein manifold; the metrics that satisfy (\ref{Einstein condition}) are Einstein metrics.
\end{definition}
In Chapter 1.1 of the book \cite{besse2007einstein}, Besse introduced an important topological property of the Einstein Manifold:
\begin{proposition}
A closed 3-dimensional (pseudo) Riemannian manifold is Einstein iff it has constant sectional curvature.
\end{proposition}

From the above proposition and Killing-Hopf (see Theorem \ref{Killing-Hopf}), we find that a 3-dimensional closed manifold that has constant sectional curvature is equivalent to a sphere $S^n.$ So it suffices to prove the Poincaré conjecture if, for a 3-manifold, we can find critical points of the Hilbert-Einstein action that gives us constant sectional curvature. 

This approach falls short because it is difficult to prove the existence of Einstein metrics that give us constant sectional curvature on any closed manifold. However, it is much easier to have a scalar curvature than having a constant sectional curvature. See the following example of a 3-manifold that does not have Einstein metric but has constant scalar curvature.

\begin{example}[$M = S^2\times S^1$]
The topological property of Einstein metrics suggest that, if a 3-manifold has an Einstein metric, then it has constant scalar curvature, its universal cover is diffeomorphic to $S^3$ or $\mathbb{R}^3$ (both Euclidean space and hyperbolic space are diffeomorphic to $\mathbb{R}^3$). Now consider $M = S^2\times S^1$ with the universal cover $S^2\times \mathbb{R}$. Since it is not homeomorphic to $S^3$ or $\mathbb{R}^3$, hence $M$ does not have an Einstein metric. 
 However, the scalar curvature of the product manifold is the sum of the scalar curvature of these two manifolds (see Proposition \ref{scalar addition}), so $M$ has constant scalar curvature if $S^2$ and $S^1$ are equipped with the standard unit sphere metric. 
\end{example}
 Therefore, it makes sense for Yamabe to consider another related problem in a more restricted area:

\paragraph{The Yamabe Problem.}
\textit{Given a compact Riemannian manifold $(M, g)$ of
dimension $n \geq 3$, is there a metric $\tilde{g}$ conformal to $g$ that has a constant scalar curvature?}\\

\textit{This survey serves as the capstone paper of my senior year under the guidance of Professor Engelstein. It follows that given in the survey \cite{Neumayer}, and the survey \cite{Lee}; for the Riemannian geometry part it mostly refers to the book \cite{CarmoR}; for the analysis part it mostly refers to the book \cite{evans}, and for the part of concentration compactness we follow the approach of Lions \cite{lions1984concentration1,lions1984concentration2}.}

\textit{We will start by taking the analytical approach to discuss how the minimizer of Yamabe functional provides constant scalar curvature, and its relationship with the Sobolev Space $W^{1,2}.$ Then with stereographic projection and dilation, we will show the importance of the sphere $S^n$, and the fact that the minimizer of Yamabe functional on standard sphere is the standard metric and its conformal diffeomorphisms. This will give us the constraint $\lambda (M) < \lambda(S^n),$ which leads us to the final theorem that the Yamabe problem is solvable when $\lambda(M) < \lambda(S^n).$ For the proof of this theorem we follow the approach of concentration compactness.}


\section{The Analytical approach of the Yamabe Problem}
Given two conformal metrics $g$ and $\tilde{g} = \phi^{4/(n-2)}g$ (for choice of index, see Appendix \ref{conformal index}). Let $R$ and $\tilde{R}$ denote the scalar curvatures of $g$ and $\tilde{g}$, respectively. These quantities are related by the identity
\begin{align}
    \tilde{R} = \phi^{1-2^*}(-c_n\Delta\phi + R\phi)\label{conformal identity}.
\end{align}\\
Recall the Hilbert-Einstein action (see Definition \ref{Hilbert-Einstein action 1}), that is, $\int_M RdV_g $. Here we consider a normalized version
\begin{align*}
Q(g) = \frac{\int_M RdVol_g}{vol_g(M)^{2/2^*}} \tag{$\star$}, \label{yamabe functional}
\end{align*}
in which $2^*=2n/(n-2)$ is chosen such that $Q(g) = Q(kg)$ for any $k>0$.

Notice that $2n/(n-2)$ is the exponent in the strict Sobolev embedding (see Appendix \ref{Sobolev Space} for the fundamental knowledge of the Sobolev space), which will play a decisive role in the latter part.\\ \\
Therefore, we define the Yamabe constant of $(M,g)$ by
\begin{align*}
\lambda(M) = \inf \left\{\frac{\int_M \tilde{R}dVol_{\tilde{g}}}{vol_{\tilde{g}}(M)^{2/2^*}}: \tilde{g}\text{ conformal to } g \right\}
\end{align*} \\
Now take the variation. For any variation $\phi\in C^{\infty}_c(M)$, we have
\begin{align*}
0 &= \frac{d}{d\epsilon}\rvert_{\epsilon=0}Q(u+\epsilon\phi) \\
   &= \frac{d}{d\epsilon}|_{\epsilon=0}\frac{\int_Mc_n\left \langle\nabla u+\epsilon \nabla \phi, \nabla u+\epsilon \nabla \phi\right \rangle + R(u+\epsilon \phi)^2dVol_g}{|u+\epsilon\phi|^2_{2^*}}\\
   &= \frac{\int_M2c_n\left \langle\nabla u, \nabla \phi\right \rangle + 2Ru\phi dVol_g}{|u|^2_{2^*}} \\
   &\ \ - 2\frac{\int_Mc_n\left \langle\nabla u, \nabla u\right \rangle + Ru^2\phi dVol_g}{|u|^2_{2^*}}\cdot\frac{\int_Mu^{2^*-1}\phi dVol_g}{|u|^2_{2^*}}\\
   &= \frac{2}{|u|^2_{2^*}}\int\phi\left(  -c_n\Delta u+Ru - \frac{u^{2^*-1}}{|u|_{2^*}^{2^*-2}}Q(u)  \right)dVol_g
\end{align*}\\
Notice that $Q(u)$ and $|u|_{2^*}$ are numbers, let $\lambda = \frac{Q(u)}{|u|_{2^*}^{2^*-2}}$ and we have $-c_n\Delta u+Ru - \lambda u^{2^*-1}=0$. If u is a minimizer, from (\ref{conformal identity}) we see that $\tilde{R} = u^{1-2^*}(-c_n\Delta u+Ru) = \lambda$ is constant.\\

To summarize thus far, to solve the Yamabe problem, it suffices to show the existence of a smooth positive minimizer of $\lambda(M)$. However, we will see in the later part that establishing such a minimizer is difficult because the problem lacks compactness.

\subsection{Why do we consider Sobolev space \texorpdfstring{$W^{1,2}$}{TEXT}}
To find a minimizer of Yamabe functional, for a minimizing sequence $\{g_i\}_{i=1}^\infty$ that satisfies $\lim\limits_{i\to\infty}Q(g_i) = \lambda (M)$, the goal is to show that $g_i$ converges to a smooth metric $\tilde{g}$, which is the minimizer of the Yamabe functional.\\

Then naturally 2 questions arise:\\

1. How do we show the convergence? \\

2. What are the properties of the limit metric?\\ \\
Whenever considering convergence of functions, we first need to choose a topology of convergence. In this case, we choose a norm on the function space. As we solve differential equations, it is natural to consider the Sobolev space $W^{k,p}$. (see \ref{Sobolev Space definition} about the definition of Sobolev space.) The question is why are we looking at space $W^{1,2}$?

Basically, this is asking that, given the basic condition of the problem, what is the largest space that $Q(u)$ could live in and $\nabla u$ must be square integrable for us to start talking about $Q(u)$.

Consider the minimizer of the normalized Hilbert-Einstein action. Its numerator is $\int_M -u\Delta u + Ru^2 dVol_g = \int_M \left \langle\nabla u, \nabla u\right \rangle + Ru^2 dVol_g$, and its denominator is $|u|_{2^*}^2$. It is natural to let the normalizing term $|u|_{2^*}^2 = 1$, and since $M$ is compact, we have $|u|_{2^*}\geq C|u|_{2}$ for some constant $C$. Therefore, $|u|_{2}$ is bounded. It follows that $|\int_M Ru^2 dVol_g| \leq sup\{|R|\}|u|_{2}^2$ is bounded. Also, since $\frac{\int_M \left \langle\nabla u, \nabla u\right \rangle + Ru^2 dVol_g}{|u|_{2^*}^2}$ is bounded, we see that $\int_M \left \langle\nabla u, \nabla u\right \rangle dVol_g$ is bounded. \\

In summary, we have $|u|_{2^*}$, $|u|_{2}$, $\int_M \left \langle\nabla u, \nabla u\right \rangle dVol_g$ all bounded. So $u$ is bounded in both $W^{1,2}(M)$ and $L^{2^*}(M)$. The Sobolev inequality (see \ref{Gagliardo-Nirenberg-Sobolev inequality} about the Sobolev inequality) tells us that $W^{1,2}(M)$ is a more restrictive condition, hence it is natural to consider $u\in W^{1,2}(M)$.

\subsection{Lack of compactness of the embedding \texorpdfstring{$W^{1,2}(M)$ into $L^{2^*}(M)$}{TEXT}}
The Yamabe functional consists two parts: the numerator $\int_M u\Delta u + Ru^2 dVol_g$, which is close to $W^{1,2}$ norm (see \ref{Sobolev norm} about Sobolev norm), and the denominator $|u|_{2^*}$. So when the Yamabe functional is bounded and approaching a value $\lambda(M)$, we can expect our function to weakly converge in the $W^{1,2}$ space, and by the Sobolev embedding theorem (see \ref{Sobolev embedding}), this also means that it will weakly converge in the $L^{2^*}(M)$ space.

Therefore we cannot just apply the direct method of the calculus of variations, since the embedding of $W^{1,2}(M)$ into $L^{2^*}$ is not compact (for compactness theorem see \ref{Rellich--Kondrachov Compactness Theorem}), and we could have lower semicontinuity of the energy only if there is strong convergence of the minimizing sequence in $L^{2^*}(M).$ \\

The following example shows that the direct method can be used to establish the existence of minimizers in a particular case of $\lambda(M) \leq 0$.
\begin{example}[Consider $\lambda(M) \leq 0$] \label{Special case of Yamabe}
\begin{align*}
E(u) = \int c_n|\nabla u|^2 + Ru^2dVol_g - \lambda(M)|u|^2_{2^*}
\end{align*}
Since $|u|^2_{2^*} \geq 0$, we have that $E(u)\geq 0$, and $E(u)=0$ iff $u$ is minimizer of the $Q(u)$.\\
Since $M$ is compact, we have the Sobolev embedding (see \ref{Sobolev embedding}):
$$W^{1,2}(M)\subset L^{2^*}(M)\subset L^2(M)$$
The first embedding is not compact. Hence if $\{u_k\}$ is a sequence with $E(u_k)$ converge to $0$, then by the Sobolev inequality, we see that $E(u_k)$ bounded implies $|u_k|_{W^{1,2}}$ bounded (we can constrain on $|u_k|_{2^*} = 1$). So there is a subsequence such that $\{u_k\}$ weakly converge in $W^{1,2}(M)$ and $L^{2^*}(M)$, and converge strongly in $L^2(M)$. Since $\lambda (M) \leq 0$, we have that
\begin{align*}
E(u) = E(\lim\limits_k u_k) \leq \lim\limits_k E( u_k) = 0.
\end{align*}
Note that if $\lambda(M) > 0$ then the above inequality may not hold, since the lower semicontinuity of the $2^*$ norm goes in the wrong direction. Also, the smoothness should relies on the regularity property of Laplacian operator.
\end{example}

From the above example we see that the lower semicontinuity in $L^{2^*}$ does not help since it is in the denominator. Another problem is that the weak limit $u$ may be equal to zero. However, we could have lower semicontinuity of energy $E(u)$ if the minimizing sequence converges strongly in $L^{2^*}(M).$  For example, if we consider the subcritical power:

\begin{example}[replace $2^*$ by any value $p < 2^*$]
By Rellich-Kondrachov theorem (see \ref{Rellich--Kondrachov Compactness Theorem}) the embedding $W^{1,2}(M) \hookrightarrow L^{2^*}(M)$ is compact, then with direct method there exist the minimizer
$$
Q_p(\phi) = \frac{\int_M (c_n\left \langle\nabla\phi,\nabla\phi\right \rangle + R\phi^2)dVol_g}{|\phi|_{p}^2}.
$$
\end{example} 

Historically, Trudinger gave this restrictive assumption in the paper \cite{Trudinger}, that the Yamabe problem could be solved whenever $\lambda(M) \leq 0.$ In fact, going one step further, he showed the existence of a positive constant $\alpha (M)$ such that the problem could be solved when $\lambda (M) < \alpha (M).$ Based on Trudinger's result, Aubin showed in the paper \cite{aubin1976equations} that $\alpha (M) = \lambda (S^n)$ for every $M.$ This established the following theorem:

\begin{theorem}[Yamabe, Trudinger, Aubin]\label{Yamabe, Trudinger, Aubin}
 Suppose $\lambda{(M)}<\lambda(S^n)$. Then there exists a minimizer of $\lambda{(M)}$ and hence a solution of the Yamabe problem on M.
\end{theorem}

\textit{This is one of the three main theorems of the Yamabe problem and serves as the main theorem of this survey. For the other two main theorems, see the survey \cite{Lee} on the proofs of $\lambda{(M)}<\lambda(S^n)$ in dimensions $3,4,5$ and higher given, respectively, by Schoen \cite{schoen1984conformal} and Aubin \cite{aubin1976equations}.}

\section{The Yamabe Problem on the Sphere}
From Theorem \ref{Yamabe, Trudinger, Aubin} we understand that the model case of the sphere $S^n$ plays an important role in the proofing of the Yamabe problem. We start this section by discussing a natural question, that is why do we consider the sphere $S^n$?

\subsection{Stereographic projection}
Let $P=(0, \ldots, 1)$ be the north pole on $S^n \subseteq \mathbb{R}^{n+1}$. Stereographic projection $\sigma: S^n\setminus P \rightarrow$ $\mathbb{R}^{n}$ is defined by $\sigma\left(\zeta^{1}, \ldots, \zeta^{n}, \xi\right)=\left(x^{1}, \ldots, x^{n}\right)$ for $(\zeta, \xi) \in S^n\setminus P$ where
$$
x^{j}=\frac{\zeta^{j}}{1-\xi}
$$
We can verify that $\sigma$ is a conformal diffeomorphism. If $\bar{g}$ is the standard metric on $S^n$, and $g_{euc}$ is the Euclidean metric on $\mathbb{R}^{n}$, then under $\sigma,$ the round metric on sphere $\bar{g}$ corresponds to
$$
(\sigma^{-1})^* \bar{g}= \frac{g_{euc}}{4\left(|x|^{2}+1\right)^2}
$$
This can be written as 
\begin{align}
4 u_{1}^{4/(n-2)} g_{euc} \text{,\ \ \ \  where   }
u_{1}(x)=\left(|x|^{2}+1\right)^{(2-n) / 2}.  \label{projection}
\end{align}\\
We denote this by $4 u_{1}^{p-2} g_{euc}$ in the latter part, where $p=2^*=2n/(n-2)$. By means of stereographic projection, it gives the conformal diffeomorphisms of the sphere induced by the standard conformal transformations on the plane, as shown in the diagram 
\[
\begin{tikzcd}
S^n \arrow[r, "\sigma"] \arrow[d, "\sigma^{-1} \delta_{\alpha} \sigma"]
& \mathbb{R}^n \arrow[d, "\delta_{\alpha}"] \\
S^n 
& \mathbb{R}^n \arrow[l,"\sigma^{-1}"]
\end{tikzcd}
\]
The group of such diffeomorphisms is generated by the rotations, together with maps of the form $\sigma^{-1} \delta_{\alpha} \sigma$ where $\delta_{\alpha}: \mathbb{R}^{n} \rightarrow \mathbb{R}^{n}$ is the dilation $\delta_{\alpha}(x)=\alpha^{-1} x$ for $\alpha\rangle 0.$ \\ \\
Combine with (\ref{projection}), we get the spherical metric on $\mathbb{R}^{n}$ transforms under dilations to
\begin{align}\label{transform under dilation}
    \delta_{\alpha}^{*} (\sigma^{-1})^{*} \bar{g}=4 u_{\alpha}^{p-2} g_{euc} \text{,\ \ \ \  where   } u_{\alpha}=\left(\frac{|x|^{2}+\alpha^{2}}{\alpha}\right)^{(2-n)/2}.
\end{align}
\\ 
For standard metric on $S^n$ we have 
\begin{align*}
    Q(u)=\lambda(S^n),
\end{align*}
where $u=1.$ Consider the dilation $\delta_{\alpha}: \mathbb{R}^{n} \rightarrow \mathbb{R}^{n}$. Since it's conformal transformation we get that
\begin{align*}
Q(u_{\alpha})=\lambda(S^n),
\end{align*}
where $u_{\alpha}$ is given by (\ref{transform under dilation}).
Notice $2-n \leq 0$, then we rewrite as $u_{\alpha}=\left(\frac{\alpha}{|x|^{2}+\alpha^{2}}\right)^{2/(n-2)}.$

Considering the two cases of whether $x$ is at the south pole, let $\alpha \rightarrow 0,$ we get that,
\begin{align*}
    u_{\alpha}&=\left(\frac{\alpha}{|x|^{2}+\alpha^{2}}\right)^{2/(n-2)} \xrightarrow{x \neq 0} 0\\
    u_{\alpha}&=\left(\frac{1}{\alpha}\right)^{2/(n-2)} \xrightarrow{x = 0} +\infty
\end{align*}
We see that for $x \neq 0, $ $u_{\alpha}$ converges weakly to 0, so all $u_{\alpha}$ concentrate near the South pole. Now consider this metric on the neighbourhood of any manifold $M$, we have
\begin{align*}
    Q_M(u_{\alpha}) &\approx Q_{S^n}(u_{\alpha})\\
    \\ &\Downarrow \\ 
    \lim_{\alpha \rightarrow 0} Q_M(u_{\alpha}) &= \lim_{\alpha \rightarrow 0}Q_{S^n}(u_{\alpha})\\
    \\ &\Downarrow \\ 
    \inf Q_M(u_{\alpha}) &\leq \lambda (S^n)\\
    \\ &\Downarrow \\ 
    \lambda(M) &<\lambda (S^n)
\end{align*}
Hence we obtain the restriction in Theorem \ref{Yamabe, Trudinger, Aubin}.
\subsection{Two important results on the sphere \texorpdfstring{$S^n$}{TEXT}}
After understanding the importance of the sphere $S^n$, we now show that the infimum of the Yamabe functional is attained by the standard metric $\bar{g}$ on the sphere $S^n.$ 

This was originally independently proved by Aubin \cite{aubin1976problemes} and G.Talenti \cite{talenti1976best}. Here, we will follow the approach by Morio Obata (\cite{obata1971conjectures}) and Karen Uhlenbeck (\cite{Sacks1981TheEO}). It consists of two parts:

\begin{enumerate}
    \item The metric $g$ is the standard metric, and it's conformal diffeomorphism. (Proposition \ref{OBATA}.)
    \item The infimum $\lambda(S^n)$ is attained by a smooth
metric $g$ in the conformal class of the standard metric $\bar{g}.$  (Proposition \ref{Ulenbeck})
\end{enumerate}
The first part is given by Obata in the survey \cite{obata1971conjectures} about the conformal diffeomorphism on the sphere.
\begin{proposition}[Obata]\label{OBATA}
If $g$ is a metric on $S^n$ that is conformal to the standard round metric $\bar{g} = \phi^{-2}g$ and has a constant scalar curvature, then up to a constant scale factor, $g$ is obtained from $\bar{g}$ by conformal diffeomorphism of the sphere.
\end{proposition}
\begin{proof}
We start by showing that $g$ is the Einstein metric (see (\ref{Einstein condition}) for the definition of the Einstein metric). By the fact that the standard metric $\bar{g}$ has a constant sectional curvature and it is an Einstein metric, we get the following.
\begin{align*}
0 = \bar{B}_{jk} =  B_{jk} + (n-2)\phi^{-1}\left( \nabla_j\nabla_k\phi + {1/n}\Delta \phi g_{jk} \right)
\end{align*}
Since $B$ is traceless, $B^{jk}g_{jk} = 0$. Then we can directly compute the norm of $B$ by
\begin{align*}
\int_{s^n} \phi |B|^2dVol_g &= \int_{s^n} \phi B_{jk}B^{jk}dVol_g\\
                                              &= -(n-2)\int_{s^n} B^{jk}\left( \nabla_j\nabla_k \phi + \frac{1}{n}\Delta\phi g_{jk}\right)dVol_g\\
                                              &= -(n-2)\int_{s^n} B^{jk}\left( \nabla_j\nabla_k \phi\right)dVol_g\\
                                              &= (n-2)\int_{s^n} \nabla_jB^{jk}\left( \nabla_k \phi\right)dVol_g\\
\end{align*}
Notice $\nabla_jB^{jk} = 0$, then $\int_{s^n} \phi |B|^2dVol_g  = 0$ and $B=0$. On the other hand, $g$ and $\bar{g}$ are both conformal to a flat metric on $\mathbb{R}^n$, and we have $W=0$. As both $W=0$ and $B=0$, the curvature tensor is
\begin{align*}
R_{ijlk} = \frac{R}{n(n-2)}(g_{ik}g_{jl} - g_{il}g_{jk})
\end{align*}
This is the same as the standard metric on $S^n$, hence by by Killing-Hopf (see Theorem \ref{Killing-Hopf}) $g$ has constant sectional curvature (not just constant scalar curvature). Therefore, $(S^n, g)$ is isometric to the standard $(S^n, \bar{g})$ by $\sigma:S^n \to S^n$, hence $g = \sigma^*\bar{g}$, and this isometry is the desired conformal diffeomorphism. 
\end{proof}

Now we have shown that the group of metrics $\bar{g}$ is a conformal diffeomorphism. However, they are not compact, that is, the family of metrics  $\delta_{\alpha}^{*} (\sigma^{-1})^{*} \bar{g}$ on the sphere are not uniformly bounded. Therefore, it is crucial to prove the existence of extremals on the sphere. This leads to the other proposition given by Uhlenbeck \cite{Sacks1981TheEO}.

\begin{proposition}\label{Ulenbeck}
There exists a positive $C^\infty$ function $\psi$ in $S^n$ that satisfies $Q_{\bar{g}}(\psi) = \lambda (S^n)$.
\end{proposition}

\begin{proof}
Without loss of generality, assume $V_g(M) = 1$.
For $2\leq s< p$ (recall $p=2^*$), let $\phi_s$ be the minimizer of $Q(\phi)/|\phi|_{s}^2 = \lambda_s$, with $|\phi_s|_s=1.$ (for this part, see \ref{yamabe propostion}). If $\phi_s$ is uniformly bounded for all $s$, then Ascoli-Arzela (see C.7 in the book \cite{evans}) implies that up to the subsequence $\phi_s$ converges to $\phi\in C^{\infty}$. Therefore, we really care about the case when $\phi_s$ is not uniformly bounded. Composing with a rotation, we may assume that each $\phi_s$ achieves its supremum at the south pole (Q), and $\phi_s(Q) \rightarrow \infty$. 
Now let $\kappa = \sigma^{-1} \delta_{\alpha} \sigma,$  we define
\begin{align*}
\psi_s = t_\alpha\kappa_\alpha^*\phi
\end{align*}
for each $s < p.$ For each $s$, choose the value of $\alpha$ so that $\psi_s(Q)=1$ at the south pole. Notice that
$$
t_\alpha = \left(\frac{(1+\xi) + \alpha^2(1-\xi)}{2\alpha}\right)^{\frac{2-n}{2}},
$$
where $t_\alpha$ is the conformal factor, and at the south pole $t_\alpha(Q) = \alpha^{(2-n)/2}$, so that for each $s$, we get 
$$
\alpha^{(2-n)/2}\phi_s(Q)=1,
$$ 
and 
$$
\lim_{s\to p}\alpha = \infty.
$$
For simplicity of computation, let us denote
\begin{align*}
\square u= -c_n\Delta u + Ru
\end{align*}
Notice $t_\alpha$ is part of the conformal translation $g_\alpha = \kappa_\alpha^*\bar{g} = t_\alpha^{p-2}\bar{g}$, this definition will make sure that, if the pull back operator $\square_\alpha = \kappa^*\square u$, then
\begin{align*}
\square_\alpha (t_\alpha^{-1}u) = t_\alpha^{1-p} \square u
\end{align*}
which implies
\begin{align*}
\int_{S^n}\square(\phi_s)\phi_s dVol_{\bar{g}} &= \int_{S^n}\square_\alpha(\kappa_\alpha^*\phi_s)\kappa_\alpha^*\phi_s dVol_{g_\alpha}\\
                                                                      &= \int_{S^n}\square_\alpha(t_\alpha^{-1}\psi_s)t_\alpha^{-1}\psi_s (t_\alpha^{p-2})^{n/2}dVol_{\bar{g}}\\
                                                                      &= \int_{S^n}\square_\alpha(t_\alpha^{-1}\psi_s)t_\alpha^{-1}\psi_s t_\alpha^pdVol_{\bar{g}}\\
                                                                      &= \int_{S^n}t_\alpha^{1-p}\square(\psi_s)t_\alpha^{-1}\psi_s t_\alpha^pdVol_{\bar{g}}\\
                                                                      &= \int_{S^n}\square(\psi_s)\psi_s dVol_{\bar{g}} 
\end{align*}
Now we also assume $\lambda(S^n) > 0$, then 
$$\int_{S^n}\square(u)u dVol_{\bar{g}} \geq \frac{\lambda(S^n)}{2}|u|_{p}\geq \frac{\lambda(S^n)}{2}|u|_{2}$$
So there is constant $C>2c_nR/\lambda(S^n)$, such that 
$$C\int_{S^n}\square(u)u dVol_{\bar{g}} \geq |u|_{1,2}$$
On the other hand it is easy to show that, there exist $C'$, with
$$\int_{S^n}\square(u)u dVol_{\bar{g}} = \int_{S^n}c_n(\langle \nabla u, \nabla u\rangle  + Ru^2)dVol_{\bar{g}} \leq C' |u|_{1,2}$$
Then we get that
\begin{align*}
|\psi_s|_{1,2}\leq C \int_{S^n}\square(\psi_s)\psi_s dVol_{\bar{g}} = C \int_{S^n}\square(\phi_s)\phi_s dVol_{\bar{g}} \leq CC' |\phi_s|_{1,2}
\end{align*}
Also direct computation as below shows that
\begin{align*}
\square(\psi_s) &= t_\alpha^{p-1}t_\alpha^{1-p}\square(\psi_s)\\
                         &= t_\alpha^{p-1}t_\alpha^{1-p}\square(t_\alpha\kappa^*\phi_s)\\
                         &= t_\alpha^{p-1}\square_\alpha(\kappa^*\phi_s)\\
                         &= t_\alpha^{p-1}\lambda_s\kappa^*\phi_s^{s-1}\\
                         &= t_\alpha^{p-s}\lambda_s\psi_s^{s-1}
\end{align*}
A brief summarize, so far we get
\begin{enumerate}
\item $\psi_s$ is bounded in $W^{1,2}(M)$
\item $\square(\psi_s) = t_\alpha^{p-s}\lambda_s\psi_s^{s-1}$
\item $\alpha_s$ goes to $\infty$ as $s\to 2^*$
\end{enumerate}
Since all the functions $\phi_s$ are smooth on $S^n$, we know that each $\psi_s$ is also smooth. Also, for any $\xi<1$ we have
$$\lim_{\alpha\to\infty}\frac{t_\alpha(\xi)}{\alpha^{\frac{2-n}{2}}} = ((1-\xi)/2)^{\frac{2-n}{2}}$$
So as $s\to p$ and $\alpha\to\infty$, $\psi_s$ is bounded by $((1-\xi)/2)^{\frac{2-n}{2}}$, which does not depend on $s$. Therefore, for each $\epsilon>0$, $K_\epsilon$ is the $\epsilon$ neighborhood of the north pole $P$, then $\psi_s$ is uniformly bounded on $S^n\setminus K_\epsilon$. This implies that up to the subsequence, $\psi_s$ converges on $C^{\infty}(S^n\setminus P)$. \\

On the other hand $\lambda_s\to\lambda(S^n)$, and $t_\alpha^{p-s}\leq 1$ for $s\to p$, so we conclude that at each point of $S^n\setminus P$, $\square \psi = f\psi$, for $f\leq \lambda(S^n)$. Also, $|\psi|_p^p \leq |\psi|_{1,2}$, so that $\psi^{p-2}\in L^{n/2}(S^n)$. By the weak removable singularities theorem (see \ref{REMOVABLE}), the singularity at point $P$ is weakly removable; hence there is $\psi\in W^{1,2}(S^n)$, weakly satisfies the equation $\square \psi = f\psi^{p-1}$. Together with 
\begin{align*}
|\psi_s|^p_p=|\phi_s|^p_p\geq V(S^n)^{1-p/s}|\phi|_s^p = 1
\end{align*}
We get that
\begin{align*}
\frac{\int_{S^n}\psi\square\psi dVol_{\bar{g}}}{|\psi_s|^p_p} &= \frac{\int_{S^n}\phi\square\phi dVol_{\bar{g}}}{|\psi_s|^p_p}\\
                                                                                                      &= \frac{\int_{S^n}\phi\square\phi dVol_{\bar{g}}}{|\phi_s|^s_p}\\
                                                                                                      &= \lambda_s
\end{align*}
Taking limit $s\to p$, we get
$$\frac{\int_{S^n}\psi\square\psi dVol_{\bar{g}}}{|\psi_s|^p_p} \leq \lambda(S^n)$$
By definition of $\lambda(S^n)$, we have that 
$$\frac{\int_{S^n}\psi\square\psi dVol_{\bar{g}}}{|\psi_s|^p_p} = \lambda(S^n)$$
This $\psi$ is what we are looking for.
\end{proof}
A notable conclusion is that $\lambda (S^n)$ is the optimal constant for Sobolev embedding in $\mathbb{R}^n$ by the above argument. This is helpful in the later part of the concentration compactness lemma (see Lemma \ref{Lions}).
\section{Proof of Theorem \ref{Yamabe, Trudinger, Aubin}}
Now we proof the main Theorem \ref{Yamabe, Trudinger, Aubin} of this note, that is, when $\lambda(M)<\lambda\left(S^{n}\right)$, we can find a minimizer of the Yamabe functional (\ref{yamabe functional}) and therefore a solution to the Yamabe problem. We follow the approach of concentration compactness by Lion \cite{lions1984concentration1,lions1984concentration2} in the survey \cite{Neumayer}, which plays an important role in the proof.



 \begin{lemma}[Lions]\label{Lions}
Suppose $\left\{u_{k}\right\}$ is uniformly bounded in $W^{1,2}(M)$, so $u_{k} \rightharpoonup u \in W^{1,2}(M)$. Up to subsequences, we can assume
$$
\begin{aligned}
\mu_{k} &:=\left|\nabla u_{k}\right|^{2} d \operatorname{vol}_{g} \stackrel{*}{\rightharpoonup} \mu, \\
\nu_{k} &:=\left|u_{k}\right|^{2^{*}} d \operatorname{vol}_{g} \stackrel{*}{\rightharpoonup} \nu .
\end{aligned}
$$
Then
\begin{equation}
\nu=|u|^{2^{*}} d \operatorname{vol}_{g}+\sum_{j \in J} \nu_{j} \delta_{p_{j}},\label{eq3.1}
\end{equation}
\begin{equation}
\mu \geq|\nabla u|^{2} d \operatorname{vol}_{g}+\sigma_{n}^{2} \sum_{j \in J} \nu_{j}^{2/{2^*}} \delta_{p_{j}}\label{eq3.2}
\end{equation}
where $J$ is an at most countable set of points in $M$.
\end{lemma}
\begin{proof}[Proof of Lemma \ref{Lions}.]
Let
$$
\begin{aligned}
v_{k} &:=\left(u_{k}-u\right) \rightharpoonup 0 \text { in } W^{1,2} \text { and } L^{2^{*}} \\
\omega_{k} &:=\left(\left|u_{k}\right|^{2^{*}}-|u|^{2^{*}}\right) d x \rightharpoonup \omega \\
\tilde{\mu}_{k} &:=\left|\nabla v_{k}\right|^{2} d x \rightharpoonup \tilde{\mu}
\end{aligned}
$$
The existence of the above weak limits are given by Banach-Alaoglu theorem (see Theorem 23.5 in the book \cite{meise1997introduction}). For $k \rightarrow + \infty,$ we have
$lim \left|u_{k}- u\right|^{2^{*}}_{2^{*}} \geq 0$, then $lim \left|u_{k}\right|^{2^{*}}_{2^{*}}-|u|^{2^{*}}_{2^{*}} \geq 0$ given by the lower semicontinuity. To show 
$$
lim \left|u_{k}- u\right|^{2^{*}}_{2^{*}} \geq 0 = lim \left|u_{k}\right|^{2^{*}}_{2^{*}}-|u|^{2^{*}}_{2^{*}} \geq 0
$$
one must show $u_k \rightarrow u$ a.e. It is not difficult since we have $u_k$ bounded in $W^{1,2}$, that is, for $2<p<2^*,$ we have $u_k \rightarrow u$ in $L^p$, which implies $u_k \rightarrow u$ a.e. Therefore, we have $\omega_{k}=\left|v_{k}\right|^{2^{*}} d x+o(1)$. Now, take any $\xi \in C_{c}^{\infty}\left(\mathbb{R}^{n}\right)$. Applying the Sobolev inequality, we have 
\begin{align}\label{equation lion}
\int \xi^{2^{*}} d \omega=\lim \int\left|\xi v_{k}\right|^{2^{*}} d x & \leq \liminf \frac{1}{\sigma_{n}^{2^{*}}}\left(\int\left|\nabla\left(\xi v_{k}\right)\right|^{2}\right)^{2^{*} / 2} 
\end{align}
If $v_k$ weakly converge to $0$ in $W^{1,2}(M)$, then for any $\xi\in C_c(M)$ we have
\begin{align*}
\int\vert\nabla (\xi v_k)\vert^2 &= \int \left\langle v_k\nabla\xi + \xi\nabla v_k, v_k\nabla\xi + \xi\nabla v_k\right \rangle \\
                                      &= \int v_k^2\vert\nabla\xi\vert^2 + \xi^2\vert\nabla v_k\vert^2 + 2\xi v_k\left \langle \nabla \xi, \nabla v_k\right \rangle
\end{align*}
And given by the strong convergence of $v_k$ in $L^2(M)$, we get
\begin{align*}
\lim\limits_{k\to\infty}\int v_k^2\vert\nabla\xi\vert^2  \leq \sup(\vert\nabla\xi\vert^2 )\lim\limits_{k\to\infty}\int v_k^2  = 0
\end{align*}
 Also given by the strong convergence of $v_k$ in $L^2(M)$, and the fact that $\vert v_k \vert_{1,2}$ is bounded for all $k$, we have
\begin{align*}
\lim\limits_{k\to\infty}\int |\xi v_k\left \langle\nabla \xi, \nabla v_k\right \rangle| &\leq \lim\limits_{k\to\infty}\int |\xi|\cdot|v_k|\cdot|\nabla \xi|\cdot|\nabla v_k|\\
                                                                                               &\leq \lim\limits_{k\to\infty}\left(\int v_k^2 \int\xi^2|\nabla \xi|^2|\nabla v_k|^2\right)^{1/2}\\
                                                                                               &\leq \lim\limits_{k\to\infty}\sup(|\xi\nabla \xi|)\left(\int v_k^2 \int|\nabla v_k|^2\right)^{1/2}\\
                                                                                               &=0
\end{align*}
So we have that 
\begin{align*}
\liminf\limits_{k\to\infty}\int|\nabla (\xi v_k)|^2 &= \liminf\limits_{k\to\infty}\int \xi^2|\nabla v_k|^2
\end{align*}

Following from \eqref{equation lion} we get that,
\begin{align*}
\int \xi^{2^{*}} d \omega =\liminf \frac{1}{\sigma_{n}^{2^{*}}}\left(\int \xi^{2}\left|\nabla v_{k}\right|^{2}\right)^{2^{*} / 2}=\frac{1}{\sigma_{n}^{2^{*}}}\left(\int \xi^{2} d \tilde{\mu}\right)^{2^{*} / 2}
\end{align*}
Rearranging the powers, we have
\begin{equation}
\sigma_{n}\left(\int \xi^{2^{*}} d \omega\right)^{2 / 2^{*}} \leq\left(\int \xi^{2} d \tilde{\mu}\right) \quad \forall \xi \in C_{c}^{\infty}\left(\mathbb{R}^{n}\right).\label{eq3.5}
\end{equation}
This very unnatural thing looks like a reverse Hölder's inequality. Applied to $\xi$ approximating the characteristic function of any open set $\Omega$, (\ref{eq3.5}) shows that $\tilde{\mu}$ controls $\omega$ nonlinearly:
\begin{equation}
\sigma_{n}^{2} \omega(\Omega)^{2 / 2^{*}} \leq \tilde{\mu}(\Omega),\label{eq3.6}
\end{equation}
This scaling will force $\omega$ to be supported on a countable set of atoms. Indeed, since $\tilde{\mu}$ is a finite measure, it contains at most countably many atoms, say at $\left\{x_{j}\right\}$. For any point $x \in B_{1} \backslash \bigcup\left\{x_{j}\right\}$, we can take any open set $\Omega$ containing $x$ with $\tilde{\mu}(\Omega) \leq \sigma_{n}^{2}$, so that \eqref{eq3.6} gives us
$$
1 \geq \sigma_{n}^{-2} \tilde{\mu}(\Omega) \geq \omega(\Omega)^{2 / 2^{*}} \geq \omega(\Omega).
$$
In other words, $\omega$ is absolutely continuous with respect to $\tilde{\mu}$ on $B_{1} \backslash \bigcup\left\{x_{j}\right\}$. 
Recall $\omega = fd\tilde{\mu} + \sum_{j \in J} \nu_{j} \delta_{x_{j}},$ by the Radon-Nikodym theorem (see \ref{Radon-Nikodym}), for $\tilde{\mu}$-a.e. $x$, we have
$$
f(x)=\lim _{r \rightarrow 0} \frac{\omega\left(B_{r}(x)\right.}{\tilde{\mu}\left(B_{r}(x)\right)} \stackrel{\eqref{eq3.6}}{\leq} \lim _{r \rightarrow 0} \sigma_{n}^{-2^{*}} \tilde{\mu}\left(B_{r}(x)\right)^{2^{*} / 2-1}=0
$$
Hence, the support of $\omega$ is contained on $\bigcup\left\{x_{j}\right\}$ and so
\begin{align}
\omega=\sum_{j \in J} \nu_{j} \delta_{x_{j}},  \label{omega} 
\end{align}
Now take any $x_{j}$ and by the reverse Hölder's inequality \eqref{eq3.5}, we have $d\tilde{\mu} \geq (d\omega)^{2/{2^*}},$ it follows that
$$
d\tilde{\mu} \geq \sum_{j \in J} \nu_{j}^{2 / 2^{*}} \delta_{x_{j}}.
$$
Now apply to a $\xi$ with $\xi\left(x_{j}\right)=1$, and $\xi=0$ on $B_{r}\left(x_{j}\right)^{c}$, we find
$$
\sigma_{n}^{2} \nu_{j}^{2 / 2^{*}} \delta_{x_{j}} \leq \tilde{\mu}\left(x_{j}\right)
$$
This gives us \eqref{eq3.2} hence concludes the proof.
\end{proof}
Now we can start proofing Theorem \ref{Yamabe, Trudinger, Aubin}.
\subsection{Proof of Theorem \ref{Yamabe, Trudinger, Aubin}}
\begin{proof}[Proof of Theorem \ref{Yamabe, Trudinger, Aubin}]
Let $\left\{u_{k}\right\}$ be a minimizing sequence for $\lambda(M)$. Without loss of generality, we may assume that $\left\|u_{k}\right\|_{2^{*}}=1$. Up to a subsequence, $u_{k} \rightarrow u$ in $L^{2}(M)$ and $u_{k} \rightharpoonup u$ in $W^{1,2}(M)$ and $L^{2^{*}}(M)$ with $\|u\|_{2^{*}}^{2^{*}}=t \in[0,1]$. Note that if $t=1$, then $u_{k} \rightarrow u$ strongly in $L^{2^{*}}$. This gives us the lower semicontinuity of the energy. Since $\{u_k\}_{k=1}^\infty$ converges weakly to $u$ in $L^{2^*}(M)$, then
$
|u|_{2^*}\leq \lim\limits_{k\to\infty}|u_k|_{2^*},
$
and also 
$$
\int_M \langle \nabla u, \nabla u \rangle dVol_g \leq \lim\limits_{k\to\infty}\int_M\langle \nabla u_k, \nabla u_k\rangle  dVol_g.
$$ 
Therefore we get
\begin{align*}
\lambda(M) &\leq \frac{\int_M\langle \nabla u, \nabla u\rangle  + Ru^2 dVol_g}{|u|_{2^*}^2}
                   = \frac{\lambda(M)-A}{1 - B},
\end{align*}
where $0\leq A\leq \lambda(M)$, and $0\leq B\leq 1$. The concentration compactness lemma \ref{Lions} implies that
\begin{align*}
A &\geq c_n\sigma_n^2\sum\limits_{j\in J}v_j^{2/2^*}\\
   &=\lambda(S^n)(\sum\limits_{j\in J} v_j)^{2/2^*}\sum\limits_{j\in J}\left(\frac{v_j}{\sum v_j}\right)^{2/2^*}\\
   &\geq \lambda(S^n)(\sum\limits_{j\in J} v_j)^{2/2^*}\sum\limits_{j\in J}\left(\frac{v_j}{\sum v_j}\right)\\
   &= \lambda(S^n)(\sum\limits_{j\in J} v_j)^{2/2^*}\\
   &= \lambda(S^n)B
\end{align*}
Now together with the assumption that $\lambda(S^n)>\lambda(M)$, we have
\begin{align*}
A > \lambda(M)B,
\end{align*}
For such relation between $A$ and $B$, we get
\begin{align*}
\lambda(M) &\leq \frac{\lambda(M)-A}{1 - B}\\
                   &\leq \frac{\lambda(M)-\lambda(M)B}{1 - B}\\
                   &= \lambda(M)
\end{align*}
Then we get $A=B=0$ and $\lambda(M) = \frac{\lambda(M)-A}{1 - B}$, which is desired. This shows the spirit of the proof, that is, to find the proper relation between A and B such that the equality holds. This exact relation is given by the concentration compactness lemma (\ref{Lions}), that is,
\begin{align}
    A = \int \tilde{\mu} = c_n\int \mu - \int |\nabla u|^2dVol_g, \label{equation A}
\end{align}
and 
\begin{align}
    B = \int \omega = \left(\int \nu - \int |u|^{2^*}dVol_g\right)^{2/2^*}. \label{equation B}
\end{align}
Therefore, by concentration compactness lemma (\ref{Lions}), we have
$$
\lambda(M)=\lim Q\left(u_{k}\right) \geq \int c_{n}|\nabla u|^{2}+R u^{2}+c_{n} \sigma_{n}^{2} \sum \nu_{j}^{2 / 2 *}
$$
Note that $\int c_{n}|\nabla u|^{2}+R u^{2} d \operatorname{vol}_{g}=t^{2 / 2^{*}} Q(u) \geq t^{2 / 2^{*}} \lambda(M)$. Recall $\lambda\left(S^{n}\right)=c_{n} \sigma_{n}^{2}$, we have
$$
\begin{aligned}
\lambda(M) & \geq t^{2 / 2^{*}} \lambda(M)+c_{n} \sigma_{n}^{2} \sum \nu_{j}^{2 / 2^*} \\
& \geq t^{2 / 2^{*}} \lambda(M)+\lambda\left(S^{n}\right)(1-t)^{2 / 2^{*}}\left(\sum \frac{\nu_{j}}{1-t}\right)^{2 / 2^{*}} \\
&=t^{2 / 2^{*}} \lambda(M)+\lambda\left(S^{n}\right)(1-t)^{2 / 2^{*}}
\end{aligned}
$$
The final equality holds because $\sum \nu_{j}=1-t$. Now, since $\lambda\left(S^{n}\right)>\lambda(M)$ and again applying Jensen's inequality, we have
\begin{align}
	\lambda(M) & \leq t^{2 / 2^{*}} \lambda(M)+\lambda\left(S^{n}\right)(1-t)^{2 / 2^{*}}\nonumber \\
	& \leq \lambda(M)\left\{t^{2 / 2^{*}}+(1-t)^{2 / 2^{*}}\right\}\label{eq3.3} \\
	& \leq \lambda(M).\label{eq3.4}
\end{align}
Equality in \eqref{eq3.4} implies that $t=0$ or $t=1$. If $t=0$, then we have a strict inequality in \eqref{eq3.3}. Therefore $t=1$. This establishes the existence of a minimizer $u \in W^{1,2}(M)$. Since $|\nabla u|=|\nabla| u||$ for a.e. $x \in M$, we may assume without loss of generality that $u \geq 0$. The results in the elliptic regularity theory (see \cite{Trudinger}) show that $u$ is smooth, and then the maximum principle ensures that $u$ is positive. Thus, our minimizer is indeed a conformal factor.
\end{proof}

\newpage

\medskip
\printbibliography 
\newpage


\appendix
\section{Basics of Riemannian Geometry}
Most of this comes from the book \cite{Yau} and the book \cite{CarmoR}. For a more fundamental knowledge of Riemannian geometry, see the book \cite{CarmoR}.

\subsection{Curvatures}
\begin{definition}[Riemann curvature]\label{Riemannian curvature}
The curvature $R$ of a Riemannian manifold $M$ is a correspondence that associates to every pair $X,Y\in \Gamma(TM)$, a mapping $R(X,Y): \Gamma(TM)\to\Gamma(TM)$ given by:
$$R(X,Y)Z = \nabla_Y\nabla_X Z - \nabla_X\nabla_Y Z + \nabla_{[X,Y]}Z$$
where $\nabla$ is the Riemannian connection of $M$.\\

The curvature measures the non-commutativity of the covariant derivative. 
\end{definition}

\begin{definition}[Sectional curvature]\label{sectional curvature}
Let $\sigma \subset  T_pM$ be a 2-dim subspace of the tangent space. Then the sectional curvature is:
$$
K(X,Y) = \frac{\left \langle R(X,Y)X,Y\right \rangle }{\left \langle X,X\right \rangle \left \langle Y,Y\right \rangle -\left \langle X,Y\right \rangle ^2},
$$
where $R$ is the Riemann curvature tensor defined above. The sectional curvature is an analog of Gauss curvature on 2-dimensional surface; it is important because the Riemannian curvature tensor is uniquely determined by the sectional curvature. 
\end{definition}

\begin{definition}[Ricci curvature and scalar curvature]\label{ricci and scalar}
Ricci curvature is the contraction of second and last index in curvature tensor:
\begin{align*}
R_{ij} &= \frac{1}{n-1}\sum\limits_{k}R_{ikj}^k\\
          &= \frac{1}{n-1}R_{ikjs}g^{sk}
\end{align*}
Scalar curvature is the contraction of Ricci curvature with the inverse of metric
\begin{align*}
K(p) = R_{ij} g^{ij}
\end{align*}
The geometric meaning of taking $i=j$ is to get the average of sectional curvatures of all the 2-d planes passing through $X_i = X_j.$
\begin{align*}
Ric(X_i, X_i) &= \frac{1}{n-1}\sum\limits_{k\neq i} \left \langle R(X_i, X_k)X_j, X_k\right \rangle \\
                    &= \frac{1}{n-1}\sum\limits_{k\neq i} K_{\sigma_{ik}}
\end{align*}
Over normal coordinate, the metric $g$ can be expressed as
\begin{align*}
g_{ij}(p) = \delta_{ij} + \frac{1}{3}R_{iklj}(p)x^kx^l + O(|x^3|)
\end{align*}
And the volume form is:
\begin{align*}
vol(p) = \sqrt{\det(g(p))} = 1 - \frac{1}{6}Ric_{kl}(p)x^kx^l + O(|x^3|)
\end{align*}
\end{definition}

\begin{proposition}\label{scalar addition}
Let $M=M_1\times M_2$ be the product of two Riemannian manifolds, and $R$ be its curvature tensor, $R_1$, $R_2$ be curvature tensor for$ M_1$ and $M_2$ respectively, then one can relate $R$, $R_1$ and $R_2$ by
\begin{align*}
R(X_1+X_2,Y_1+Y_2,Z_1+Z_2,W_1+W_2)=R_1(X_1,Y_1,Z_1,W_1)+R_2(X_2,Y_2,Z_2,W_2)
\end{align*}
where $X_i,Y_i,Z_i,W_i\in TM_i$.
\end{proposition}
To show this, we need the following:
\begin{enumerate}
    \item $\left \langle X_1+X_1,Y_1+Y_2\right \rangle =\left \langle X_1,Y_1\right \rangle M_1+\left \langle X_2,Y_2\right \rangle M_2;$
    \item $[X_1+X_2,Y_1+Y_2]M=[X_1,Y_1]M_1+[X_2,Y_2]M_2;$
    \item $\nabla^{M}_{X_1+X_2}(Y_1+Y_2)=\nabla^{M_1}_{X_1}(Y_1)+\nabla^{M_2}_{X_2}(Y_2).$
\end{enumerate}
Here part 1 is simply by definition of product Riemannian manifold, part 2 can be shown in local coordinates, and part 3 can be shown by part 1 and part 2 and along with Koszul formula. For more details, see Exercise 1(a) of Chapter 6 in the book \cite{CarmoR}.

\section{The Conformal Map \texorpdfstring{$\tilde{g} =\phi^{4/(n-2)}g$}{TEXT}}\label{conformal index}

First, to prove the above definition, we have $P(g) = P(kg)$. Assume $\tilde{g} = kg$, where $k$ is a constant. The standard scalar curvature transformation formula for $\tilde{g} = e^{2\phi}g$ is
\begin{align}\label{standard scalar transformation}
\tilde{R} = e^{-2\phi }R-2(n-1)e^{-2\phi }\Delta \phi -(n-2)(n-1)e^{-2\phi }|d\phi |^{2}.
\end{align}
Since $e^{2\phi} = k$ is a constant, we have that $\Delta \phi = 0$, and $d\phi = 0.$ Therefore, $\tilde{R} = e^{-2\phi }R = R/k$, on the other hand
\begin{align*}
Vol_{\tilde{g}} &= \sqrt{\det\tilde{g}}dx\\
                        &= k^{n/2}\sqrt{\det g}dx\\
                        &= k^{n/2}Vol_g\\
\end{align*}
So we get the conformal transformation for Hilbert action
\begin{align*}
\int_M \tilde{R}dVol_{\tilde{g}} &= \int_M k^{-1}Rk^{n/2}dVol_{g}\\
                                                   &= \int_M Rk^{(n-2)/2}dVol_{g}\\
                                                   &= k^{(n-2)/2}\int_M RdVol_{g}
\end{align*}
And the transformation of volume is
\begin{align*}
V(\tilde{g}) &= \int_M dVol_{\tilde{g}}\\
                   &= k^{n/2}\int_M dVol_{g}\\
                   &= k^{n/2}V(g)\\
\end{align*}
Therefore, for $P(g) = \frac{\int_M RdVol_g}{V(g)^\alpha}$, we have 
\begin{align*}
P(kg) &= \frac{k^{(n-2)/2}\int_M RdVol_g}{(k^{n/2})^\alpha V(g)^\alpha}\\
         &= \frac{k^{(n-2)/2}}{(k^{n/2})^\alpha}P(g)\\
         &= P(g)
\end{align*}
Since we need $k^{(n-2)/2} = k^{n\alpha/2}$, we have $\alpha = \frac{n-2}{n} = 2/2^*$, it follows that $2^*=\frac{2n}{(n-1)}.$ \\

Now for more general cases let $\tilde{g} = Ag$, in which $A$ is a smooth function, then we have
\begin{align*}
V(\tilde{g}) &= \int_M dVol_{\tilde{g}}\\
                   &= \int_M A^{n/2}dVol_{g}
\end{align*}
So 
\begin{align*}
P(\tilde{g}) &= \frac{\int_M\tilde{R}dVol_{\tilde{g}}}{\left(\int_M A^{n/2}dVol_{g}\right)^{2/2^*}}\\
                   &= \frac{\int_M\tilde{R}dVol_{\tilde{g}}}{\left(\left(\int_M A^{n/2}dVol_{g}\right)^{\frac{1}{2^*}}\right)^2}\\
                   &= \frac{\int_M\tilde{R}dVol_{\tilde{g}}}{\left(\left(\int_M (A^{\frac{n-2}{4}})^{2^*}dVol_{g}\right)^{\frac{1}{2^*}}\right)^2}\\
                   &= \frac{\int_M\tilde{R}dVol_{\tilde{g}}}{\left(|A^{\frac{n-2}{4}}|_{2*}\right)^2}
\end{align*}
Denote $\phi = A^{\frac{n-2}{4}}$, then the denominator becomes $|\phi|_{2*}^2$. Then by formula (\ref{standard scalar transformation}), we also get
\begin{align}\label{conformal_scalar_curvature}
\tilde{R} = \phi^{1-2^*}(-c_n\Delta\phi + R\phi)
\end{align}
And that
\begin{align*}
\int_M\tilde{R}dVol_{\tilde{g}} &= \int_M\phi^{1-2^*}(-c_n\Delta\phi + R\phi)\phi^{2^*}dVol_g\\
                                                  &= \int_M\phi(-c_n\Delta\phi + R\phi)dVol_g\\
                                                  &= \int_M (-c_n\phi\Delta\phi + R\phi^2)dVol_g\\
                                                  &= \int_M (c_n\left \langle\nabla\phi, \nabla\phi\right \rangle  + R\phi^2)dVol_g\\
\end{align*}
At last we conclude this note by
\begin{align*}
P(\phi) = \frac{\int_M \tilde{R}dVol_{\tilde{g}}}{vol_{\tilde{g}}(M)^{2/2^*}} =  \frac{\int_M (c_n\left \langle\nabla\phi,\nabla\phi\right \rangle  + R\phi^2)dVol_g}{|\phi|_{2^*}^2}.
\end{align*}


\section{Sobolev Space}\label{Sobolev Space}
Most of the materials refer to Chapter 5 of the book \cite{evans}.
\begin{definition}[Sobolev Space]\label{Sobolev Space definition}
Let $1 \leq p \leq \infty$, and let $k \geq 0$ be a natural number. A function $f$ is said to lie in $W^{k,p}(\mathbb{R}^n)$ if its weak derivatives $D f$ exist and lie in $L^p(\mathbb{R}^n)$ for all $j=0,\ldots,k$. If $f$ lies in $W^{k,p}(\mathbb{R}^n)$, we define the $W^{k,p}$ norm of $f$ by the formula
$$ \|f\|_{W^{k,p}(\mathbb{R}^n)} := \sum_{j=0}^k \|Df\|_{L^p(\mathbb{R}^n)}$$.
\end{definition}

\begin{definition}[Sobolev norm] \label{Sobolev norm}
If $u \in W^{k,p}(M)$, we define its norm to be
\begin{equation}
{\lVert u \rVert}_{W^{k,p}(M)}=\left\{
\begin{aligned}
(\sum\!_{\lvert\alpha \rvert \leq k}\int_{M}\vert D^{\alpha}u \vert^p dx)^{\frac{1}{p}}\ \ \ (1\leq p<\infty) \\
\sum\!_{\lvert\alpha \rvert \leq k} \text{ess sup}\textsubscript{M} \vert D^{\alpha}u \vert\ \ \ (p=\infty) .
\end{aligned}
\right.
\end{equation}
\end{definition}

\begin{note}[Sobolev conjugate of p]\label{sobolev conjugate}
If $1\leq p < n$, the Sobolev conjugate of p is
\begin{align*}
    p^*=\frac{np}{n-p}.
\end{align*}
\end{note}
\begin{theorem}[Gagliardo-Nirenberg-Sobolev inequality]\label{Gagliardo-Nirenberg-Sobolev inequality}
Assume $1\leq p < n.$ There exists a constant $C$ depending only on $p$ and $n$, such that
\begin{align}
{\lVert u \rVert}_{L^{p^*}(\mathbb{R}^n)} \leq C{\lVert Du \rVert}_{L^{p}(\mathbb{R}^{n})}
\end{align}
for all $u\in C^{1}_{c}(\mathbb{R}^n).$
\begin{proof}
Assume $p=1$.Since u has compact support, for each $i=1,...,n$ and $x \in \mathbb{R}^n$ we have
\begin{align*}
    u(x)=\int_{- \infty}^{x_i} u_{x_i}(x_1,...x_{i-1},y_i,x_{i+1},...,x_n)dy_i,
\end{align*}
then
\begin{align*}
    \lvert u(x)\rvert \leq \int_{-\infty}^{\infty} \lvert Du(x_1,...x_{i-1},y_i,x_{i+1},...,x_n)\rvert dy_i,
\end{align*}
it follows that
\begin{align}\label{integral1}
    \lvert u(x)\rvert^{\frac{n}{n-1}} \leq \prod_{i=1}^{n} \left( \int_{-\infty}^{\infty} \lvert Du(x_1,...,y_i,...,x_n)\rvert dy_i\right)^{\frac{1}{n-1}}
\end{align}
Integrate (\ref{integral1}) with respect to $x_1,...,x_n$ and using the generalized Hölder inequality, we have
\begin{align*}\label{p=1}
    \int_{\mathbb{R}^n}\lvert u(x)\rvert^{\frac{n}{n-1}} dx 
    &\leq \prod_{i=1}^{n} \left( \int_{-\infty}^{\infty} \int_{-\infty}^{\infty}...\int_{-\infty}^{\infty}\lvert Du\rvert dx_1...dy_i...dx_n\right)^{\frac{1}{n-1}}\\
    &=\left( \int_{\mathbb{R}^n}\lvert Du\rvert dx\right)^{\frac{n}{n-1}}
\end{align*}
From Note \ref{sobolev conjugate} we know $\frac{n}{n-1}$ is the Sobolev conjugate $p^*$ when p = 1, it follows that
\begin{align}
    \int_{\mathbb{R}^n}\lvert u\rvert^{p^*} dx \leq \left( \int_{\mathbb{R}^n}\lvert Du\rvert dx\right)^{p^*}
\end{align}
Now consider the case of $1<p<n$. Choose $\vert v\vert^{\frac{n}{n-1}}=\vert u \vert^{p^*}$, we have
\begin{align*}
    \int_{\mathbb{R}^n} \vert v\vert^{\frac{n}{n-1}} dx
    &\leq \left( \int_{\mathbb{R}^n}\lvert Dv\rvert dx\right)^{\frac{n}{n-1}} \\
    &=\left( \int_{\mathbb{R}^n}\lvert Du\rvert^{\frac{(n-1)p^*}{n}} dx\right)^{\frac{n}{n-1}} \\
    &=\left( \int_{\mathbb{R}^n}\lvert Du\rvert^{\frac{(n-1)}{n}\frac{np}{n-p}} dx\right)^{\frac{n}{n-1}} \\
    &=\left( \int_{\mathbb{R}^n}\lvert Du\rvert^{\frac{(n-1)p}{n-p}} dx\right)^{\frac{n}{n-1}}
\end{align*}
By the definition of derivative we have
\begin{align*}
    \lvert Du\rvert^{\frac{(n-1)p}{n-p}} 
    &= \lvert u^{\frac{(n-1)p}{n-p}-1}\rvert\lvert Du\rvert \\
    &= \lvert u^{\frac{n(p-1)}{n-p}}\rvert\lvert Du\rvert.
\end{align*}
Apply Hölder's inequality, it follows that
\begin{align*}
     \int_{\mathbb{R}^n} \vert v\vert^{\frac{n}{n-1}}dx
     &\leq \left(\int_{\mathbb{R}^n}\lvert u^{\frac{n(p-1)}{n-p}}\rvert\lvert Du\rvert dx\right)^{\frac{n}{n-1}}\\
     &= \left(\int_{\mathbb{R}^n}\lvert u ^{\frac{n(p-1)}{n-p}\cdot \frac{p}{p-1}}\rvert^{\frac{p-1}{p}} \left( \lvert Du\rvert^p\right)^{\frac{1}{p}} dx \right)^{\frac{n}{n-1}}\\
     &= \left(\int_{\mathbb{R}^n}\lvert u ^{\frac{np}{n-p}}\rvert dx \right)^{\frac{p-1}{p}\cdot{\frac{n}{n-1}}} \left(\Vert Du\Vert_{L^p} \right)^{\frac{n}{n-1}}\\
     &= \left(\int_{\mathbb{R}^n}\lvert u \rvert^{p^*} dx \right)^{\frac{n(p-1)}{p(n-1)}} \left(\Vert Du\Vert_{L^p}\right)^{\frac{n}{n-1}}
\end{align*}
Recall $\vert v\vert^{\frac{n}{n-1}}=\vert u \vert^{p^*}$, then
\begin{align*}
    \int_{\mathbb{R}^n}\vert u \vert^{p^*}dx \leq \left(\int_{\mathbb{R}^n}\lvert u\rvert ^{p^*} dx \right)^{\frac{n(p-1)}{p(n-1)}} \left(\Vert Du\Vert_{L^p}\right)^{\frac{n}{n-1}}
\end{align*}
It follows that
\begin{align*}
     \left(\int_{\mathbb{R}^n}\vert u \vert^{p^*}dx \right)^{1-\frac{n(p-1)}{p(n-1)}} 
     =\left(\int_{\mathbb{R}^n}\vert u \vert^{p^*}dx \right)^{\frac{n-p}{p(n-1)}}
     &\leq \left(\Vert Du\Vert_{L^p}\right)^{\frac{n}{n-1}}
\end{align*}
and so
\begin{align*}
    \left(\int_{\mathbb{R}^n}\vert u \vert^{p^*}dx \right)^{\frac{n-p}{np}}
     \leq \left(\Vert Du\Vert_{L^p}\right)
\end{align*}
Notice $\frac{n-p}{np} = \frac{1}{p^*}$, therefore we have
\begin{align*}
    \left(\int_{\mathbb{R}^n}\vert u \vert^{p^*}dx \right)^{\frac{1}{p^*}}
     \leq C\left(\int_{\mathbb{R}^n}\vert Du\vert^{p}dx \right)^{\frac{1}{p}}
\end{align*}
\end{proof}
\end{theorem}

\begin{proposition}[Estimates for $W^{1,p},1 \leq p < n$.] \label{Sobolev embedding}
Let M be a bounded open subset of $\mathbb{R}^{n}$, suppose $\partial M$ is $C^1$. Assume $1 \leq p < n$ and $u \in W^{1,p}(M)$. Then $u\in L^{p^*}(M)$, with the estimate
\begin{align*}
    {\lVert u \rVert}_{L^{p^*}(M)} \leq C{\lVert u \rVert}_{W^{1,p}(M)}
\end{align*}
The constant $C$ depending only on $p$ and $n$,and $M$.
\end{proposition}

\begin{theorem}[Sobolev embedding theorem for one derivative] 
Let $1 \leq p \leq q \leq \infty$ be such that $\frac{n}{p}-1 \leq \frac{n}{q}$. Then $W^{1,p}(M)$ embeds continuously into $L^q(M)$.
\end{theorem}

\begin{theorem}[Rellich--Kondrachov Compactness Theorem] \label{Rellich--Kondrachov Compactness Theorem}
Assume $U$ is a bounded open subset of $\mathbb{R}^n,$ and $\partial M$ is $C^1$. Let $1 \leq p \leq q \leq \infty$ be such that $\frac{n}{p}-1 < \frac{n}{q}$. Then 
$$
W^{1,p}(M) \subset\subset L^q(M).
$$
\end{theorem}

\section{Analytic preliminaries}
\begin{theorem}[Weak removable singularities theorem]\label{REMOVABLE}
Let $U$ be an open set in $M$ and $P\in U$. Suppose $u$ is a weak solution of $(\Delta + h)u = 0$ in
$U - {P}$, with $h \in L^{n/2}(U)$ and $u \in L^q(U)$ for some $q > p/2 = n/(n - 2)$.Then $u$ satisfies $(\Delta + h)u = 0$ weakly on all of $U$.
\end{theorem}
For the proof of this theorem, check Proposition 2.7 in the survey \cite{Lee}.

\begin{proposition}[Yamabe]
\label{yamabe propostion}
For $2 \leq s \leq p,$ there exists a smooth, positive solution $\phi_s$ to the subcritical equation, for which $Q^s(\phi_s) = \lambda_s,$ and $\vert \phi_s \vert_s = 1.$
\end{proposition}
For the proof of this proposition see the paper \cite{yamabe} and the survey \cite{Lee}.

\begin{theorem}[Radon-Nikodym] \label{Radon-Nikodym}
On the measurable space $(X,\Sigma)$, define two $\sigma-$finite measures, $\mu$ and $\nu$. It states that if $\nu \ll \Sigma$ (that is, if $\nu$ is absolutely continuous with respect to $\mu$, then there exists a $\sigma-$measurable function $f:X\to [0,\infty ),$ such that for any measurable set  $\{\Omega \subseteq X\},$
$$
\nu(\Omega) = \int_\Omega fd\nu. 
$$
\end{theorem}

\end{document}